\documentclass[11pt]{amsart}
\usepackage{amsmath}
\usepackage{amssymb}
\usepackage{amsfonts}
\usepackage{graphicx}
\usepackage{epsfig}
\usepackage[all]{xy}

\theoremstyle{plain}
\newtheorem{theorem}{Theorem}
\newtheorem{corollary}[theorem]{Corollary}
\newtheorem{lemma}[theorem]{Lemma}
\newtheorem{proposition}[theorem]{Proposition}

\theoremstyle{definition}
\newtheorem{definition}{Definition}
\newtheorem{remark}[theorem]{Remark}

\newtheoremstyle{citing}
  {3pt}
  {3pt}
  {\itshape}
  {}
  {\bfseries}
  {.}
  {.5em}
  {\thmnote{#3}}

\theoremstyle{citing}

\newcommand{\RR}{\mathbb{R}}

\newcommand{\CC}{\mathbb{C}}

\newcommand{\NN}{\mathbb{N}}

\newcommand{\ZZ}{\mathbb{Z}}

\newcommand{\QQ}{\mathbb{Q}}

\newcommand{\sS}{\mathbb{S}}

\def\B{{\mathcal B}}

\def\E{{\mathcal E}}

\def\M{{\mathcal M}}

\def\P{{\mathcal P}}

\begin{document}
\title{Eigenvalues and strong orbit equivalence.}

\author[M.I. Cortez, F. Durand, and S. Petite ]{Mar\'{\i}a Isabel Cortez$^{1,3}$,  Fabien Durand$^{2}$, and Samuel Petite$^{2}$}

\address{$^1$Departamento de Matem\'atica y Ciencia de la Computaci\'on, Facultad de Ciencia, Universidad de Santiago de Chile, Av. Libertador Bernardo O'Higgins 3363, Santiago, Chile.}
 
\address{$^2$Laboratoire Ami\'enois
de Math\'ematiques Fondamentales et Appliqu\'ees, CNRS-UMR 7352, Universit\'{e} de Picardie Jules Verne, 33 rue Saint Leu, 80000 Amiens, France.} 

\address{$^3$F\'ed\'eration de Recherche ARC Math\'ematiques, CNRS-FR 3399, Universit\'{e} de Picardie Jules Verne, 33 rue Saint Leu, 80000 Amiens, France.\vspace{3mm}} 
 
 \thanks{M. I. Cortez was partially funded by Anillo Research Project 1103 DySyRF and Fondecyt Research Project 1140213. She thanks the hospitality of the LAMFA UMR 7352 CNRS-UPJV and the "poste rouge" CNRS program. }

\email{maria.cortez@usach.cl}
\email{fabien.durand@u-picardie.fr}
\email{samuel.petite@u-picardie.fr}

\maketitle{}

\begin{abstract}
We give conditions on the subgroups of the circle to be realized as the subgroups of eigenvalues of minimal Cantor systems  belonging to a determined strong orbit equivalence class. Actually, the additive group of continuous eigenvalues  $E(X,T)$ of the minimal Cantor system $(X,T)$ is a subgroup of the  intersection $I(X,T)$ of all the images of the dimension group by its traces. We show, whenever the infinitesimal subgroup of the dimension group associated to $(X,T)$ is trivial,  the quotient group $I(X,T)/E(X,T)$ is torsion free. We give examples with non trivial infinitesimal subgroups where this property fails.  We also provide some realization results. 
\end{abstract}

\section{Introduction}
Two dynamical systems are  orbit equivalent if there is a  bijection between their phase spaces that preserves their structures (measure preserving, topological, etc.) and  induces a one-to-one correspondence between their orbits.  
The notion of orbit equivalence arises first in the context of probability measure preserving group actions (measurable orbit equivalence),  as a consequence of the study of von Neumann algebras \cite{Murray&vonNeumann:1936, Singer:1955}.  One of the most remarkable results  in this theory establishes that there is only one orbit equivalence class among the free ergodic  probability measure preserving  actions of amenable groups  \cite{Dye:1959, Ornstein&Weiss:1980}.

Motivated by the measurable orbit equivalence results, in particular, the characterization of the   orbit equivalence classes in terms of von Neumann algebras \cite{Krieger:1969, Krieger:1976},   Giordano, Putnam and Skau  obtain in \cite{Giordano&Putnam&Skau:1995} one of the most important results in the context of the orbit equivalence from a topological point of view: the orbit equivalence classes of the minimal $\ZZ$-actions on the Cantor set are characterized in terms of the $K_0$ group of the associated $C^*$-algebra (see \cite{Renault:2009,Tomiyama:1987} for an interplay between $C^*$-algebras and dynamics).  As a consequence, they obtain that there are as many orbit equivalent classes as reduced simple dimension groups with distinguished order unit. 
Thus, unlike the measurable setting, in the topological context it is natural to ask for the dynamical properties which are preserved under orbit (or strong orbit) equivalence. For instance, in  \cite{Herman&Putnam&Skau:1992}  it is shown that the set of invariant probability measures of a given minimal Cantor system is affinely isomorphic to the set of traces of the associated dimension group. Thus the set of invariant probability measures is preserved, up to affine homeomorphism, under strong orbit equivalence. On the contrary, within a strong orbit equivalence class it is possible to find a minimal Cantor systems having any possible entropy (see \cite{Boyle&Handelman:1994},  \cite{Ormes:1997} and \cite{Sugisaki:2003} for the general case).

In this paper, we study the relation between (strong) orbit equivalence and the spectral properties of a system. We know from \cite{Ormes:1997} that strong orbit equivalent minimal Cantor systems share the same subgroup of rational continuous eigenvalues. Thus, if the subgroup of rational continuous eigenvalues of  a minimal Cantor system  is not cyclic, then within its strong orbit equivalence class there is no mixing minimal Cantor systems.
It is no longer true for the orbit equivalence as shown again in \cite{Ormes:1997}. 
Indeed, Ormes proved (Theorem 8.2 in \cite{Ormes:1997}) that in a prescribed orbit equivalence class it is possible to realize  any countable subgroup of the circle as a group of measurable eigenvalues. 

In this work we investigate the case of non-rational eigenvalues and whether the dimension groups induce restrictions (other than those due to the rational eigenvalues) on the groups of eigenvalues that can be realized within this given strong orbit equivalence class.

It happens that a first restriction has been shown in \cite{itza-ortiz:2007} : the additive group of eigenvalues, $E(X,T)$, of a minimal Cantor system $(X,T)$, is a subgroup of the intersection of all the images of the dimension group by its traces. 
Dynamically speaking, it is a subgroup of $I (X,T) = \cap_{\mu \in \mathcal{M} (X,T)} \left\{ \int f d\mu | f \in C (X,\mathbb{Z}) \right\}$, where $\mathcal{M} (X,T)$ is the set of $T$-invariant probability measures of $(X,T)$ and $C (X,\mathbb{Z})$ is the set of continuous functions from $X$ to $\mathbb{Z}$.
An other proof of this observation can  be found in \cite{Cortez&Durand&Host&Maass:2003} but it was not pointed out.

In this paper we prove the following strong restriction.

\begin{theorem}
\label{theo:main}
Suppose that $(X,T)$ is a minimal  Cantor system  such that the infinitesimal subgroup of  the dimension group $K^0(X,T)$ is trivial.  
Then the quotient group  $I(X,T)/E(X,T)$ is torsion free.
\end{theorem}

To illustrate this result, take $K^0 (X,T) = \mathbb{Z} + \alpha \mathbb{Z} = I(X,T)$, with $\alpha$ irrational. This is the case for a Sturmian subshift. Then within the strong orbit equivalence class of $(X,T)$ the only groups of continuous eigenvalues that can be realized are $\mathbb{Z}$, which will provide topologically weakly mixing minimal Cantor systems, and $\mathbb{Z} + \alpha \mathbb{Z}$.
Moreover, both can be realized, in the first case using results in \cite{Ormes:1997} and in the second case it is realized by a Sturmian subshift. 

Relations between additive eigenvalues and topological invariants can be found in \cite{Schwartzman:1957,Riedel:1982,Packer:1986,Exel:1987}, but they do not apply to Cantor systems. 

\medskip

In Section \ref{definitions} we recall the concept of Kakutani-Rohlin partitions that will be necessary through this paper. The next section is concerns the notions and definitions we will need.
In particular, we recall the algebraic notions and dynamical interpretations  of dimension group, trace, infinitesimal and rational subgroup.
Section \ref{section:main} is devoted to the  proof of our main result: Theorem \ref{theo:main}. 
To this aim we use a precise description of entrance times with respect to some well-chosen Kakutani-Rohlin partitions.  We follow the approach proposed in \cite{Bressaud&Durand&Maass:2005,Bressaud&Durand&Maass:2010} to tackle eigenvalue problems.
Apart from Theorem \ref{theo:main}, there are three results that could be of independent interest. 
We obtain a new necessary condition to be an eigenvalue (Proposition \ref{prop:tendtozero}). 
We give an elementary proof of the fact that $E(X,T)$ is included in $I(X,T)$ (Proposition \ref{prop:eigincludedinimage}).
For every $\alpha$ in $I(X,T)$, we show there exists a continuous function $f :X \to \mathbb{Z}$ such that $\alpha = \int f d\mu$ for all $T$-invariant measure $\mu$ (Lemma \ref{lemma 0}).

In the last section we provide realization examples around Theorem \ref{theo:main}.

\section{Definitions and background}
\label{definitions}

\subsection{Dynamical systems}

We introduce  here the notations and recall some classical facts. We refer to \cite{Petersen:1983} for a more detailed expository. 
By a {\it topological dynamical system}, we mean a couple $(X,T)$
where $X$ is a compact metric space and $T: X \to X$ is a homeomorphism.
We say that it is a {\it Cantor system} if $ X $ is a Cantor space;
that is, $ X $
has a countable basis of its topology which consists of closed and
open sets (clopen sets) and does not have isolated points. It is {\em minimal} if it does not contain any non empty proper closed  $T$-invariant subset. 
A dynamical system $(Y,S)$ is called a {\em factor}  of $(X,T)$ if there is a continuous 
and onto map $\phi: X \rightarrow Y$, called a {\em factor map}, such that $\phi \circ T = S \circ \phi$. 
If $\phi$ is one-to-one we say that $\phi $ is a {\em conjugacy} and that $(X,T)$ and
$(Y,S)$ are {\em conjugate}. If $(X,T)$ is minimal and $\phi:X\to Y$ is a factor map for which there exists $x\in X$ such that $\sharp\phi^{-1}(\phi(x))=1$, we say that  $\phi$ is an {\it almost 1-1 factor map} and $(X,T)$ is an {\it almost 1-1 extension} of $(Y,S)$. 

We denote by $\mathcal{M} (X,T)$ the set of all $T$-invariant probability measure $\mu$, defined on
the Borel $\sigma$-algebra $\B_X$ of $X$. 
For such a measure $\mu$, the  quadruple $(X, \B, \mu_X, T)$ is called a {\em measurable dynamical system}.
This system is called {\em ergodic} if any $T$-invariant measurable set has  measure $0$ or $1$.
Two measurable dynamical systems $(X, \B, \mu, T)$ and $(Y, \B', \nu, S) $  are {\em measure theoretically conjugate} if we can find  invariant subsets $X_{0} \subset X$, $Y_{0} \subset Y$ with $\mu(X_{0}) = \nu(Y_{0}) =1$ and a bimeasurable  bijective map $\psi \colon X_{0} \to Y_{0}$ such that $S \circ \psi = \psi \circ T$ and $\mu(\psi^{-1} B) = \nu(B)$ for any $B \in \B'$.

A complex number $\lambda$ is a {\it continuous eigenvalue} (resp. a {\em measurable eigenvalue})
of $(X,T)$ if there exists a continuous (resp. integrable with respect to an invariant measure)  function $f : X\to \CC$,
$f\not = 0$, such that
$f\circ T = \lambda f$; $f$ is called a {\it continuous
eigenfunction}, associated to $\lambda$ . Of course any continuous eigenvalue is a measurable one for any fixed measure. Hence, every eigenvalue is of modulus 1, i.e.,  belongs to the circle $\sS^1 = \{ \lambda \in \mathbb{C}; \ |\lambda | = 1 \}$, and every
eigenfunction has a constant modulus. Notice that any continuous eigenfunction provides  a factor map from $(X,T)$ to a rotation. 

In this work we are mainly concerned with continuous eigenvalues $\lambda = \exp (2i\pi \alpha )$ of minimal Cantor systems. 
Such $\alpha$ is call an {\em additive continuous eigenvalue} of $(X,T)$, and the set of  additive continuous eigenvalue, denoted $E(X,T)$,  is an additive subgroup of $\mathbb{R}$ called the {\em group of additive continous eigenvalues}.
It is well-known that $E (X,T)$ is  countable and contains $\ZZ$.

We say two dynamical systems $(X,T)$ and $(Y,S)$ are {\em orbit equivalent} (OE) whenever there exists a homeomorphism $\phi : X \to Y$ sending orbits to orbits: for all $x\in X$, 
$$
\phi \left(\{ T^n x \mid  n\in \mathbb{Z} \} \right) = \{ S^n \phi (x) \mid  n\in \mathbb{Z} \}.
$$

This induces the existence of maps $\alpha : X\to \mathbb{Z}$ and $\beta : X \to \mathbb{Z}$ satisfying:
for all $x\in X$,
$$
\phi \circ T (x) = S^{\alpha (x)}\circ \phi (x) \hbox{ and } \phi \circ T^{\beta (x)} (x) = S\circ \phi (x).
$$

When $\alpha$ and $\beta$ have both at most one point of discontinuity, we say $(X,T)$ and $(Y,S)$ are {\em strongly orbit equivalent} (SOE)\index{strongly orbit equivalent!SOE}\index{orbit equivalent!strongly}. 
We recall below the seminal result in \cite{Giordano&Putnam&Skau:1995} that characterized these equivalence in terms of dimension groups.

\begin{theorem}
\cite{Giordano&Putnam&Skau:1995}
\label{theo:GPS}
Let $(X,T)$ and $(Y,S)$ be two minimal Cantor dynamical systems.
The following are equivalent:
\begin{enumerate}
\item
$(X,T)$ and $(Y,S)$ are strong orbit equivalent.
\item
$K^0(X,T)$ and $K^(Y,S)$ are isomorphic as dimension groups with order units.
\end{enumerate}
The following are also equivalent:
\begin{enumerate}
\item
$(X,T)$ and $(Y,S)$ are orbit equivalent.
\item
$K^0(X,T)/{\rm Inf} (K^0 (X,T))$ and $K^(Y,S)/{\rm Inf} (K^0 (Y,S))$ are isomorphic as dimension groups with order units.
\end{enumerate}
\end{theorem}

\subsection{Partitions and towers}
\label{subsec:partandto}

Sequences of partitions associated to minimal Cantor
systems were used in \cite{Herman&Putnam&Skau:1992} to build representations
of such systems as adic transformations on ordered
Bratteli diagrams. Here we do not introduce the whole formalism of
Bratteli diagrams since we will only use the language describing
the tower structure, even if both languages are very close.
We recall some definitions and fix some notations.

For    a minimal Cantor system $(X,T)$, a {\it clopen Kakutani-Rohlin partition} (CKR partition) is a
partition $\P$ of $X$ given by
\begin{equation}
\label{eq:def-KR}
\P = \{ T^{j} B(k); \  1\leq k\leq C , \ 0 \leq j< h(k) \},
\end{equation}
where $C, h(1),\dots,h(k)$ are positive integers, and 
$B(1),\dots , B(C)$ are clopen subsets of $X$ such that 
$$ \bigcup_{k=1}^C T^{h(k)} B(k)=  \bigcup_{k=1}^C B(k).$$

The set $B =\bigcup_{1\leq k\leq C}B(k)$ is called  the {\it base} of $\mathcal{P}$.
Let
\begin{equation}
\label{eq:def-seq-KR}
\bigl\{
\P_n=
\{
T^{j}B_{n} (k); 1\leq k\leq C_n,\ 0\le j<h_{n}(k)
\}
\bigr\}_{ n\in\mathbb{N}}
\end{equation}
be a sequence of CKR partitions. For every $n\in \NN$, we denote
$B_n$ the base of $\P_n$.
To be coherent with the notations of  \cite{Bressaud&Durand&Maass:2005}, we assume that $\P_0$ is the
trivial partition, that is, $B_0=X$, $C_0=1$ and $h_{0} (1) = 1$, and for the partition $\P_{1}$, $h_{1}(k) =1$ for any integer $1 \le k \le C_{1}$. 

We say that  the sequence $\{\P_n\}_{n\in \NN}$ is \emph{nested} if it satisfies: for any integer
$n\in \NN$

\medskip

{\bf (KR1)} $B_{n+1} \subseteq B_n$;

\medskip

{\bf (KR2)} $\P_{n+1} \succeq \P_n$; i.e., for any $A\in \P_{n+1}$
there exists an atom  $A^{'}\in \P_n$ such that $A\subseteq A^{'}$;

\medskip

{\bf (KR3)} $ \bigcap_{n\in \NN} B_n $ consists of a unique point;

\medskip

{\bf (KR4)} the sequence of partitions spans the topology of $X$.

\medskip

In \cite{Herman&Putnam&Skau:1992} it is proven that given a minimal Cantor system
$(X,T)$, there exists a nested sequence of CKR partitions fulfilling
{\bf (KR1)}--{\bf (KR4)} with the following
additional technical conditions: for any integer $n\geq 0$,

\medskip

{\bf (KR5)}  for any $1\leq k \leq C_{n}$, $1\leq l\leq C_{n+1}$,
there exists an integer $0 \leq j < h_{n+1} (l)$ such that $T^{j} B_{n+1}(l)
\subseteq B_{n}(k)$;

\medskip

{\bf (KR6)}  $B_{n+1} \subseteq B_{n}(1)$.

\medskip

We associate to the sequence $\{\P_n\}_{n \in \NN}$,  
the sequence of matrices $\{M_n\}_{n\ge 1}$, where
$M_n = (m_n(l,k))_{ 1\leq l \leq C_n , 1\leq k \leq C_{n-1}}$ is
given by
$$
m_n(l,k) =
\#
\{
0 \leq j < h_{n}(l) ; \  T^{j} B_{n}(l) \subseteq B_{n-1}(k)
\}.
$$
Notice that {\bf (KR5)} is equivalent to:
for any $n\ge 1$, the matrix $M_n$ has  positive entries.

For $n\geq 0$, we set $H_n=(h_{n}(l) ; 1\leq l\leq C_n)^T$.
Since the sequence of partitions is nested, we have 
$H_n=M_nH_{n-1}$ for any $n\geq 1$. 
Notice also that, by the convention,  
\begin{align}
\label{h1}
M_{1}=H_1 = (1, \dots , 1)^T.
\end{align}

For $n>m\geq 0$, we define
\begin{align}\label{h2}
P_{n,m}=M_nM_{n-1}\dots M_{m+1}, P_{1} = M_{1}, \hbox{ and } P_{n+1}=P_{n+1,1}.
\end{align}

Clearly, we have the relations
\begin{align}\label{h3}
P_{n,m}(l,k)=\#\bigl\{ 0\leq j < h_n(l);\
T^{j}B_n(l)\subseteq B_m(k) \bigr \},
\end{align}

for $1 \leq l \leq C_n, \ 1 \leq k \leq C_m$, and
\begin{align}\label{h4}
P_{n+1,m}H_m= H_{n+1}=P_{n+1}H_1.
\end{align}

Along the paper, we will strongly use a technique that we call {\em telescoping}: That is, starting from a sequence of CKR partitions  $\{ \P_{n}\}_{n\in \NN}$ fulfilling  {\bf (KR1)}--{\bf (KR6)}, we will consider an infinite subsequence of partitions satisfying an additional property. Actually, it is plain to check that any such subsequence of CKR partition satisfies also {\bf (KR1)}--{\bf (KR6)}. Moreover, the sequences of the associated matrices of the type $\{M_{n}\}_{n\in \NN}$ and $\{P_{n,m}\}_{n>m \ge 0}$ are subsequences of the previous ones. 

\subsection{Dimension groups, traces, infinitesimals and rational subgroups}
\label{subsec:dimtraceserg}

\subsubsection{Dimension groups}
We recall here some basic definitions of the algebraic  notion of dimension groups arising from the $C^*$-alegbras. Relations with dynamical systems will be explain in the next subsection. Most of the notions arises from  \cite{Effros:1981}.  
 
By an {\it ordered group} we shall mean a countable abelian group $G$ together with a subset $G^+$, called the {\em positive cone}, such that $G^+-G^+ = G$, $G^+\cap (-G^+) = \{0\}$ and $G^+ + G^+ \subset G^+$.
We shall write $a\leq b$ if $b-a\in G^+$. 
We say that an ordered group is \emph{unperforated} if $a\in G$ and $na\in G^+$ for some $a\in G$ and $n\in\mathbb{N}$ implies that $a\in G^+$. 
Observe that an unperforated group is torsion free. 
We say $(G, G^+)$ is \emph{acyclic} whenever $G$ is not isomophic to $\mathbb{Z}$.
By an \emph{order unit} for $(G,G^+)$ we mean an element $u\in G^+$ such that for every $a\in G$, $a\leq nu$ for some $n\in \mathbb{N}$. 

\begin{definition}
A \emph{dimension group} $(G,G^+,u)$ with distinguished order unit $u$ is an unperforated ordered group $(G,G^+)$ satisfying the \emph{Riesz interpolation property}, i.e., given $a_1, a_2, b_1, b_2 \in G$ with $a_i\leq b_j$ ($i,j=1,2$),
there exists $c\in G$ with $a_i\leq  c \leq b_j$, $(i,j=1,2)$.
\end{definition}

We say that two dimension groups $(G_1,G_1^+,u_1)$ and $(G_2,G_2^+,u_2)$ are isomorphic whenever there exists an order isomorphism $\phi : G_1 \to G_2$, i.e., $\phi$ is a group isomorphism such that $\phi(G_1^+)=G_2^+$, and
$\phi(u_1)=u_2$. An {\em order ideal} is a subgroup $J$ such that $J= J^+-J^+$ (where $J^+ = J\cap G^+$) and $0\leq a\leq b \in J$ implies $a\in J$.
A dimension group $(G,G^+,u)$ is \emph{simple} if it contains no non-trivial order ideals. 
It is easily seen that $(G,G^+)$ is a simple dimension group if and only if every $a\in G^+\backslash \{0\}$ is an order unit. 
Moreover, an unperforated simple ordered group is acyclic if and only if it satisfies the Riesz interpolation property (see \cite{Effros&Handelman&Shen:1980}).
Thus, the dimension groups are all acyclic.

\subsubsection{Traces}

Let $(G,G^+,u)$ be a simple dimension group with distinguished order unit $u$. 
We say that a homomorphism $p : G \to \mathbb{R}$ is a {\em trace} (also called  a \emph{state}) if $p$ is non negative (i.e., $p(G^+)\geq 0$) and $p(u)=1$. We denote the
collection of all traces on $(G,G^+,u)$ by $S(G,G^+,u)$. 
Now $S(G,G^+,u)$ is a convex compact subset of the locally convex space $\mathbb{R}^G$ endowed with the product topology. In fact, one can show that $S(G, G^+ , u)$ is a Choquet simplex. 
It is a fact (see \cite{Herman&Putnam&Skau:1992}) that $S(G,G^+, u)$ determines the order on $G$. Actually,
$$
G^+ = \{ a\in G; \ p(a)>0, \forall p\in S(G, G^+,u)\} \cup \{0\} .
$$

As we will see later, the following group is fundamental in the study of continuous eigenvalues of minimal Cantor systems.

\begin{definition} 
Let $(G,G^+,u)$ be an ordered group with unit. 
We call {\it image subgroup } of $(G,G^+,u)$ the subgroup of $\RR$
given by

$$
I(G,G^+,u)=\bigcap_{\tau\in S(G,G^+,u)}\tau(G).
$$ 
\end{definition}

\subsubsection{Infinitesimals}
Let $(G,G^+)$ be a simple dimension group and let $u\in G^+\backslash\{0\}$. We say that an element $a\in G$ is \emph{infinitesimal} if $-\epsilon u \leq a\leq \epsilon u$ for all $0<\epsilon\in \mathbb{Q}^+$ (for
$\epsilon = \frac{p}{q},\, p,q\in \mathbb{N}$, then  $a\leq \epsilon u$ means that $qa \leq pu$). 

It is easy to see that the definition does not depend upon the particular order unit $u$.
An equivalent definition is: $a\in G$ is infinitesimal if $p(a)=0$ for all $p\in S(G,G^+,u)$. The collection of infinitesimal elements of $G$ forms a  subgroup, \emph{the infinitesimal subgroup of $G$}, which we denote by $\mathrm{Inf}(G)$.

Observe that the quotient group $G/\mathrm{Inf}(G)$ is also a simple dimension group for the induced order, and the infinitesimal subgroup of $G/\mathrm{Inf}(G)$ is trivial (see \cite{Herman&Putnam&Skau:1992}). 
Furthermore, an order unit for $G$ maps to an order unit for  $G/\mathrm{Inf}(G)$.
Moreover the traces space of $G$ and $G/{\rm Inf} (G)$ are isomorphic.

When  $S(G,G^+,u)$ consists of a unique trace,  notice that 
$G/{\rm Inf}(G)$ is isomorphic to $(I(G,G^+,u), I(G,G^+,u)\cap
\RR^+,1)$, as ordered groups with unit.

\subsubsection{Rational subgroups}

By a \emph{rational group} $H$ we shall mean a subgroup of $\mathbb{Q}$ that contains $\mathbb{Z}$. We say that $H$ is a \emph{cyclic} rational group if $H$ is isomorphic to $\mathbb{Z}$. Clearly $(H,H\cap\QQ^+,1)$
is a simple dimension group with distinguished order unit $1$.
For a  simple dimension group with order unit $(G,G^+,u)$, we define the {\em rational subgroup} of $G$, denoted $\mathbb{Q}(G,G^+,u)$ (or $\mathbb{Q}(G,u)$ for short), by
$$
\mathbb{Q}(G,u) = \{ m/n; \ n\in \mathbb{N}^*, m\in \mathbb{Z},   \exists g\in G, \ ng=mu \}.
$$

The notion of rational subgroup of a dimension group with distinguished order unit depends heavily upon the choice of the order unit. 

Notice that for $n, m \in \ZZ$ and  $g\in G$ such that $ng=mu$, one gets, for any trace $\tau$, $\tau (g) = m/n$.
Consequently, $\mathbb{Q}(G,u)$ is a subgroup of $I (G, G^+, u)$.

\subsection{Dynamical interpretation of dimension groups, traces, infinitesimals and rational subgroups}

We consider here $(X,T)$  a minimal Cantor dynamical system.

\subsubsection{``Dynamical'' dimension groups}
We denote by $C (X,\mathbb{Z})$ the set of continuous maps from $X$ to $\mathbb{Z}$.
Consider the map $\beta : C (X,\mathbb{Z}) \to C (X,\mathbb{Z})$ defined by 
$\beta f =f\circ T - f$ for all $f\in C (X,\mathbb{Z})$.
The images of $\beta$ are called {\em coboundaries}.
Let $H(X,T)$ be the quotient group $C(X,\mathbb{Z})/\beta C(X,\mathbb{Z})$. The class of  a function $f \in C(X, , \ZZ)$ in this quotient is denoted by $[f]$.
We call {\em order unit}  the class $[1]$ of the constant function equal to $1$.

The positive cone, $H^+(X,T,\mathbb{Z})$, is the  set of classes  of non-negative functions $C(X,\NN )$.
Finally, the triple 
$$
K^0 (X,T) = (H(X,T,\mathbb{Z}) , H^+(X,T,\mathbb{Z}),[1] ). 
$$

is an ordered group with order unit. 
It is moreover a dimension group and, which is less immediate, a converse also holds.

\begin{theorem}
\cite{Herman&Putnam&Skau:1992}
\label{theo:hps}
If $(X,T)$ is a minimal Cantor system, then $K^0 (X,T)$ is a simple dimension group. Furthermore, if $(G,G^+ , u)$ is a simple dimension group then there is a minimal Cantor system $(X,T)$ such that $K^0 (X,T)$ and $(G,G^+ , u)$ are isomorphic.
\end{theorem}

\subsubsection{Traces are invariant measures}
Given any invariant probability measure $\mu$ of the system $(X,T)$, we associate a trace $\tau_{\mu}$ on $K^0(X,T)$ defined  by $ \tau_{\mu} ([f]) := \int f d\mu$ for any $f \in C(X,\ZZ)$. It is shown in \cite{Herman&Putnam&Skau:1992}, that the map $\mu \mapsto \tau_{\mu}$ is an affine isomorphism from  the space of $T$-invariant probability measures $\mathcal{M} (X,T)$ to the traces space $S(K^0 (X,T))$.

We denote by $I(X,T)$ the image subgroup $I(K^0(X,T))$. 
Rephrasing the definition of the image subgroup in dynamical terms, it is clear that 
$$
I(X,T) = \bigcap_{\mu\in \mathcal{M} (X,T)} \left\{ \int f d\mu; f\in C(X,\mathbb{Z}) \right\}.
$$ 


\subsubsection{Infinitesimals are functions with zero integral for all invariant measures}
\label{sec:infint}
We have seen that ${\rm Inf} (K^0 (X,T)) = \{ g\in K^0 (X,T); \tau (g) = 0 \hbox{ for all traces } \tau \}$.
Thus, due to the identification described before, we also have
$$
{\rm Inf} (K^0 (X,T)) = \left\{ [f] \in K^0 (X,T); \int f d\mu = 0 \hbox{ for all  } \mu \in \mathcal{M} (X,T) \right\} .
$$

Observe that if $(X,T)$ is uniquely ergodic, then
$K^0(X,T)/{\rm Inf}(K^0(X,T))$ is isomorphic to $(I(X,T), I(X,T)\cap
\RR^+,1)$, as ordered groups with unit.

\subsubsection{The rational subgroup is the group of rational continuous eigenvalues}
\label{sec:rationaleig}
For any  $m/n \in \mathbb{Q} (K^0 (X,T))$,  there exists a class $[f]$ such that $n[f] =m [1_X]$ and thus $\int f d\mu = m/n$ for any $T$-invariant measure $\mu$. Thus, we have the inclusion
$$ \QQ(K^0(X,T))\subset I(X,T) \cap \QQ.$$

The following theorem gives a clear dynamical interpretation of such elements.

\begin{theorem}
\cite{Giordano&Putnam&Skau:1995,Ormes:1997}
\label{theo:gpso}
Let $(X,T)$ be a minimal Cantor system and let $\mu$ be any $T$-invariant measure. Then, a rational $\frac pq$ is an additive continuous eigenvalue of $(X,T)$, i.e.,  belongs to $E(X,T)$, if and only if $\frac pq = \int f d\mu$ for some $f\in C(X,\mathbb{Z})$.
Or, equivalently, 
$$
E(X,T)\cap \mathbb{Q} =  \QQ(K^0(X,T)).
$$
\end{theorem}

As observe in \cite{Giordano&Putnam&Skau:1995} (see also \cite{Ormes:1997}), this implies the following result.

\begin{corollary}
Let $(X,T)$ and $(Y,S)$ be two strong orbit equivalent minimal Cantor systems (i.e., $K^0 (X,T)$ is isomorphic to $K^0 (Y,S)$).
Then, $(X,T)$ and $(Y,S)$ share the same rational continuous eigenvalues,  that is
$$E (X,T)\cap \mathbb{Q} = E (Y,S)\cap \mathbb{Q}. $$
\end{corollary}

\section{Group of eigenvalues and image of traces}
\label{section:main}

In this section $(X,T)$ stands for a given minimal Cantor system.
We fix a sequence $\{\mathcal{P}_n\}_n$ of CKR partitions of $(X,T)$ satisfying ({\bf KR1})-({\bf KR6}) and \eqref{h1}.
We recall  such a partition always exists.
Once it is fixed, we freely use the notations of Section \ref{subsec:partandto}.

\subsection{Some necessary conditions to be an eigenvalue}\label{sec:necessarycond}

The following results are fundamental in our study of eigenvalues of minimal Cantor systems. We set also some notations.
\begin{lemma}
\cite[Theorem 3, Theorem 5]{Bressaud&Durand&Maass:2010}
\label{lemma:thelemma}
Let $\alpha \in E(X,T)$. 
Then, there exist an integer $m >1 $, a real vector $v_m$ and an integer vector $w_m$ 
such that
\begin{enumerate}
\item
$\alpha P_mH_1 = v_m + w_m$, and
\item $\sum_{n> m}\|P_{n,m} v_m\|_{\infty}<\infty$,
\end{enumerate}
where $\Vert \cdot \Vert_{\infty}$ denotes the supremum norm. 
\end{lemma}
For any given  $T$-invariant probability measure $\mu$ of $(X,T)$ we set 
$$
\mu_n = (\mu (B_n(k))^T_{1\leq k \leq C_n} 
$$
and we call it the {\em measure vector} of $(X,T)$.
It is easy to check it fullfills, for all $1\leq m<n$, the following identities:
\begin{align}
\label{eq:vectormeasure}
\mu_1^T H_1= 1 \hbox{ and }
\mu_m^T = \mu_n^T P_{n,m} .
\end{align}

\begin{lemma}\label{lemma:orthogonalite}\cite{Bressaud&Durand&Maass:2010}
With the conditions and the notations of Lemma \ref{lemma:thelemma}, for any $T$-invariant probability measure $\mu$ of $(X,T)$, for  any integer $n\ge m$,
$$\alpha =  \mu_{m}^T  w_{m}  \hspace{0.5cm} \textrm{ and } \hspace{0.5cm} 0 =  \mu_{n}^T P_{n,m} v_{m}.$$
\end{lemma}
\begin{proof}
For any integer $n> m>1$, we set $v_n=P_{n,m}v_m$ and $w_n=P_{n,m}w_m$.
Observe that Relation \eqref{eq:vectormeasure} and  Lemma \ref{lemma:thelemma} imply for any integer $n> m$, 
$$
\alpha
=
\alpha\mu_1^T H_1=\mu_m^TP_mH_1\alpha
=
\mu^T_m w_m+\mu^T_mv_m
=
\mu_m^T w_m+\mu_n^TP_{n,m}v_m .
$$
Since $\alpha-\mu_m^Tw_m$ does not depend on $n$ and $\lim_{n\to\infty} \Vert \mu_n^TP_{n,m}v_m \Vert_{\infty }=0$, we deduce that $\mu_m^Tv_m = \mu_n^TP_{n,m}v_m=0$, for every $n> m$.
\end{proof}
The following proposition and lemma will provide key arguments in the proof of our main result Theorem \ref{theo:main}.
For its proof we need to introduce some crucial quantities as it can be seen in the series of papers \cite{Cortez&Durand&Host&Maass:2003, Bressaud&Durand&Maass:2005, Bressaud&Durand&Maass:2010} and \cite{Durand&Frank&Maass:2014}.

For any  integer $n\in \mathbb{N}$, we define the {\em  entrance time}  $r_n(x)$ of  a point $x\in X$ to the base  $B_n $ by $r_n(x)=\min\{ j\geq 0; T^jx \in B_n  \}$.
The {\it suffix map of order $n$} is the map  $s_n:X\to \mathbb{N}^{C_n}$ given by
$$
(s_n(x))_k=\sharp
\{ j\in \NN; \ 0 \le j < r_{n+1}(x),   T^j x \in B_{n} (k) \}
$$ 
for every  $k\in \{1, \ldots,  C_n\}$.
A classical computation gives\footnote{Observe that by the conventions on $\P_{0}$ and $\P_{1}$, we have  $s_{0}(x) = 0$ for any $x\in X$.} (see  for example \cite{Bressaud&Durand&Maass:2005})
\begin{align}
\label{e:formulareturn}
r_n(x)= \sum_{k=1}^{n-1}
\langle s_k(x),P_kH_1 \rangle,
\end{align}
where $\langle v, v' \rangle = v^Tv'$ stands for the usual scalar product.
\begin{proposition}
\cite{Bressaud&Durand&Maass:2005}
\label{prop:condcont} Let $(X,T)$ be a minimal Cantor system and let  $\alpha \in \mathbb{R}$. 
The following conditions are equivalent,

\begin{enumerate}
\item
$\alpha$ belongs to $E(X,T)$;
\item
the sequence of functions $(\exp (2i\pi \alpha r_n(\cdot)))_n$ converges uniformly.
\end{enumerate}
\end{proposition}

\begin{proposition}
\label{prop:tendtozero} 
Let $(X,T)$ be a Cantor minimal system.
Let $\lambda=\exp(2i\pi\alpha)$ be a continuous
eigenvalue of $(X,T)$. 
Then, 
$$
\max_{x\in X} |\langle  s_n(x) ,||| \alpha P_n H_1||| \rangle | \to_{n\to +\infty} 0,
$$
where $||| \cdot |||$ denotes the distance to the closest integer vectors.  
\end{proposition}

\begin{proof}
The sequence $(|||\alpha r_n |||)$ is a uniform Cauchy sequence (Proposition \ref{prop:condcont}) and $\alpha (r_{n+1} - r_n ) =   \langle s_n , \alpha P_n H_1 \rangle $. Therefore   $ ||| \langle s_n , \alpha P_n H_1 \rangle |||$   converges to zero.
Using the notation in the proof Lemma \ref{lemma:orthogonalite} we have $\alpha P_n H_1 = v_n + w_n$.
Consequently, there exists $n_0$ such that for all $n\geq n_0$ and all $x$ we get 

$$
||v_n || < \epsilon  < 1/8 \hbox{ and } ||| \langle s_n (x) , v_n \rangle ||| < \epsilon .
$$

We may write $\langle s_n (x) , v_n \rangle = \epsilon_n (x) + E_n(x)$ with $|\epsilon (x)| < \epsilon$ and $E_n(x)$ an integer vector. 
Notice that $(\epsilon_n)_n$ converges uniformly to $0$.

Consider the set $A = \{ x \in X ; E_n(x) = 0 \}$.
Observe that $\bigcap_{n}B_n$ is contained in $A$ so it is non empty. It is not difficult to check that  $A$ is closed. 
Let us check it is $T$-invariant. 
We fix some $x \in A$. It is straightforward to verify that there are only three possible cases for $s_n(T(x))$ : $s_n (Tx)=s_n(x)$, $s_n (T(x)) = 0$ and $s_n (T(x))= s_n (x)-e$ for some vector $e$ from the canonical base. 
The first two cases are easy to handle.
For the last one, consider

\begin{align*}
|E(Tx ) - E (x)| = & |\epsilon (x) -\epsilon (Tx) + \langle s_n (T x) ,  v_n \rangle - \langle s_n (x) ,  v_n \rangle |\\
 \leq  & \frac 14 + |\langle e ,  v_n \rangle | \leq \frac 14 +  ||v_n|| \leq \frac 12 .
\end{align*}

Therefore $E(Tx) = E(x) = 0$.
By minimality we obtain that $A=X$ which implies that 
$$
\langle  s_n(x) ,||| \alpha P_n H_1||| \rangle  = \langle  s_n(x) ,v_n \rangle = \epsilon_n (x) .
$$
This achieves the proof.
\end{proof}

\begin{lemma}
\label{the lemma}
Let $(X,T)$ be a minimal Cantor system. Then, for any  $k\in \mathbb{Z}^{*}$ and $\alpha \in \mathbb{R}$ such that $k\alpha $ belongs to $E(X,T)$ and $\alpha$ does not,  there exist an integer  $m>1$, a real vector $v_m$ and an integer vector $w_m$ 
such that 
\begin{enumerate}
\item
$k\alpha P_m H_1 = w_m + v_m$,
\item $\sum_{n> m}\| P_{n,m}v_m\|_{\infty}<\infty$,
\item for every  measure $\mu \in \M(X,T)$ and integer $ n\ge m$, $\langle  \mu_{n},  P_{n,m} v_{m} \rangle =0$.  
\item the vector $\frac 1k P_{n,m} w_{m}$ is not an integer vector for infinitely many integers $n >m$.
\end{enumerate}
\end{lemma}

\begin{proof}
From Lemma \ref{lemma:thelemma} and Lemma \ref{lemma:orthogonalite}, there exist a positive integer $m$, a real vector $v_m$ and an integer vector $w_m$ satisfying the items (1), (2) and (3).
From Proposition \ref{prop:condcont}, the sequence $(k\alpha r_n )_{n}$ converges uniformly $(\mod  \mathbb{Z})$.
Moreover, the relation \eqref{e:formulareturn} gives us for any integer $n> m+1$
\begin{align*}
\alpha r_n (x) = & 
\displaystyle \sum_{i= 1}^{n-1}  \alpha \langle  s_i(x) , P_i H_1 \rangle  \\
= & \displaystyle \sum_{i=1}^m \langle  s_i(x) , \alpha P_i H_1 \rangle 
+
\sum_{i = m+1}^{n-1}   \langle  s_i(x) , P_{i,m} \left(\frac 1k {v_m} +\frac 1k w_{m} \right)\rangle.
\end{align*}

Suppose that $\frac 1k P_{i,m}w_m$ is an integer vector for any  large enough integer  $i$. To obtain a contradiction, by Proposition \ref{prop:condcont}, it suffices to show that $(\alpha r_n )_{n}$ is a Cauchy sequence.
As $(k\alpha r_n)_{n}$ converges uniformly ($\mod \mathbb{Z}$), we deduce that $(\sum_{i = m+1}^{n}   \langle  s_i(x) , P_{i,m} {v_m} \rangle )_{n}$ converges uniformly ($\mod \mathbb{Z}$). Hence, given  $\epsilon \in (0, \frac 12)$, there exists an integer $n_{0}$ such that for any integer $n \ge n_{0}$, any integer $p \ge 0$ and $x \in X$, there exists an integer $E_{p}(x)$ such that 
$$  \vert \sum_{i=n}^{n+p} \langle s_{i}(x), P_{i,m}v_{m} \rangle - E_{p}(x) \vert < \frac{\epsilon}4.$$

By Proposition \ref{prop:tendtozero}, we can assume that the integer $n_{0}$ is sufficiently large to have 
\begin{align}\label{eq:prop10}
 \max_{x\in X} |\langle  s_n(x) ,  P_{n,m}v_{m} \rangle | < \frac {\epsilon}4 \hspace{0.5cm} \forall n \ge n_{0}.
\end{align}
Now fix $n\geq n_0$. Notice  that 
\begin{eqnarray*}
E_{p+1}(x) - E_{p}(x) =& E_{p+1}(x) - \sum_{i=n}^{n+p+1} \langle s_{i}(x), P_{i,m}v_{m} \rangle 
                                     -\left(  E_{p}(x) - \sum_{i=n}^{n+p} \langle s_{i}(x), P_{i,m}v_{m} \rangle \right) \\
                                    & +\langle s_{n+p+1}(x), P_{n+p+1,m} v_{m} \rangle.  
 \end{eqnarray*}
We deduce  that $\vert E_{p+1}(x) - E_{p}(x) \vert < \epsilon < \frac 12$ and so $E_{p+1}(x) = E_{p}(x)$ for any $x\in X$, and  $p\ge 0$. 

The inequality \eqref{eq:prop10} ensures that $E_{0}(x) =0$, and thus $E_{p}(x)= 0$ for any $p \ge 0$. It follows that  $(\sum_{i = m+1}^{n}   \langle  s_i(x) , P_{i,m} {v_m} \rangle )_{n}$ is a uniform Cauchy type sequence in $x$,  so the sequence $(\alpha r_{n})_{n}$ converges uniformly $(\mod \ZZ)$. This gives a contradiction.
\end{proof}

\subsection{Group of eigenvalues versus image group of dimension group}

A fundamental fact for this work is the following proposition (Proposition \ref{prop:eigincludedinimage}).
This has been previously shown in \cite{itza-ortiz:2007}, but has been also obtained in \cite{Cortez&Durand&Host&Maass:2003} (Proposition 11) without to be claimed.

\begin{proposition}
\label{prop:eigincludedinimage} Let $(X,T)$ be a minimal Cantor system. 
Then the set of additive continuous eigenvalues $E(X,T)$ is a subgroup of the image subgroup  $I(X,T)$.  
\end{proposition}

\begin{proof}
It suffices to show that $E(X,T) $ is a subset of $I (X,T)$.
From Lemma \ref{lemma:orthogonalite},  there exist a positive  integer $m$ and a vector $w_m\in \ZZ^{C_m}$  such that for every invariant measure $\mu$ one gets 
$$
\alpha=\mu_m^Tw_m=\int f d\mu,
$$

where $f=\sum_{k=1}^{C_m}w_m(k)1_{B_m (k)}$. 
This shows that $\alpha$ is in $I(X,T)$. 
\end{proof}



For any $\alpha \in I (X,T)$, by definition, for every invariant measure $\mu$ there exists $f_{\mu} \in C (X,\mathbb{Z})$ (thus depending on $\mu$) such that $\alpha = \int f_\mu d\mu$. 
In the next lemma, we show this function can be chosen independently of the invariant measures. 
 
\begin{lemma}
\label{lemma 0}
Let $(X,T)$ be a minimal Cantor system. If $\alpha $ belongs to  the image subgroup $I(X,T)$, then there exists a function $g\in C(X , \mathbb{Z})$ such that $\int g d\mu = \alpha$, for any measure $\mu \in \mathcal{M} (X,T)$.
\end{lemma}

\begin{proof}
If $\mathcal{M} (X,T)$ is a singleton then the result is obvious.  
From now, we will assume that $\mathcal{M} (X,T)$ contains at least two elements.
For any $g\in C(X,\mathbb{Z})$, we define
$$
\mathcal{M}_g=\left\{\mu \in \mathcal{M} (X,T); \  \int g d\mu = \alpha \right\}.
$$

Observe that $\mathcal{M}_g$ is convex and closed with respect to the weak$^*$ topology in $\mathcal{M} (X,T)$.  
From the definition of $I (X,T)$, it is clear we have
$$
\mathcal{M} (X,T) = \bigcup_{g\in C(X,\mathbb{Z})} \mathcal{M}_g.
$$
Since $C(X,\mathbb{Z})$ is countable, Baire's theorem implies there exists a map  $g_{0}\in C(X,\mathbb{Z})$ such that $\mathcal{M}_{g_{0}}$ has a non empty interior. 
It follows that  $I \colon \mu \mapsto \int g_{0}d\mu$ is an affine map which is constant on an open set of $\M (X,T)$.

We get the conclusion by showing this map is  constant. To prove this, let $\mu_0$ be in the interior of $\mathcal{M}_{g_{0}}$, and let $\mu_{1}$ be another measure in $\M (X,T)$. 
The map $t \in [0,1] \mapsto I(t\mu_{0} +(1-t) \mu_{1}) \in \RR$ is an affine map taking at least two times the same value $\alpha$. So, it is a constant map and $I(\mu_{0}) = I(\mu_{1}) = \alpha$. Since the measure $\mu_{1}$ is arbitrary, this concludes the proof.
\end{proof}

\begin{remark}\label{Glasner-Weiss}
{\rm Observe that from \cite[Lemma  2.4]{Glasner&Weiss:1995}, Lemma \ref{lemma 0} implies that for any $\alpha\in I(X,T)\cap (0,1)$ there exists a clopen set $U$ such that $\alpha=\mu(U)$ for any $T$-invariant probability measure $\mu$. In particular this is true when $\alpha$ is in $E(X,T)$.
}
\end{remark}

By Theorem \ref{theo:hps}, this lemma can of course be rephrased in terms of dimension group $(G,G^+,u)$. 
Let $\tilde{G}$ denote the group $\{g \in G; \ \tau(g) = \tau (g') \textrm{ for every traces } \tau, \tau' \in S(G, G^+,u) \}$. Notice that the unit $u$ and any infinitesimal in  $\mathrm{Inf}(G)$ belong to $\tilde{G}$.
 
\begin{corollary}\label{cor:onto}
Let $(G,G^+,u)$ be a simple dimension group. Then for any trace $\tau \in S(G,G^+,u)$, the morphism
$$ \tau \colon \tilde{G} \to I(G,G^+,u) $$ is a surjective order preserving morphism. 
In particular the dimensions groups $(\tilde{G} /\mathrm{Inf}(G),\tilde{G}\cap G^+/\mathrm{Inf}(G), [u])$, with $[u]$ the class of the unit, and $(I(G,G^+, u), I(G,G^+,u) \cap \RR^+, 1)$ are isomorphic. 
\end{corollary}

We are now able to prove our main theorem (Theorem \ref{theo:main}).

\begin{proof}[Proof of Theorem \ref{theo:main}]
Suppose $I(X,T)/E(X,T)$ is not torsion free. 
Then, there exist $\alpha \in I(X,T)\setminus E(X,T)$ and an integer $k>1$ such that  $k\alpha\in E(X,T)$. 
From Lemma \ref{the lemma}, this implies there exist a positive  integer $m$, vectors  $w_m\in \ZZ^{C_m}$ and $v_m\in \RR^{C_m}$ such that $k\alpha P_mH_1=v_m+w_m$. We recall, we set $v_n=P_{n,m}v_m$ and $w_n=P_{n,m}w_m$ for every $n\geq m$, and  the vector $v_{n}$ is orthogonal to the vector of measures  $\mu_{n}$. 

Since $\alpha$ is not a continuous eigenvalue, Lemma \ref{the lemma} implies
there must be infinitely many $n$'s such that the following set is not empty:
$$
I_n=\{i\in\{1,\cdots, C_n\}: w_n(i) \mbox{ is not divisible by } k\}. 
$$
Telescoping the sequence of CKR partitions if needed, we can assume that $I_n\neq \emptyset$ for every sufficiently large $n$. For $1\leq i\leq C_n$, we write
$$
w_n(i)=ka_n(i)+b_n(i), 
$$
where  $a_{n}(i ) \in \ZZ$ and  $b_n(i)$ is some integer in $\{0,\ldots, k-1\}$. 
Thus, the index $i$ is in   $I_n$ if and only if $b_n (i)\neq 0$. 
Observe that for any $n \ge m$
\begin{equation}\label{eq1}
\alpha=\alpha\mu_1^TH(1)=\mu_n^TP(n)H(1)\alpha=\frac{1}{k}\mu_n^T\left(v_n+w_n\right)=\frac{1}{k}\mu_n^Tw_n=\frac{1}{k}\mu_n^Tb_n+ \mu_n^Ta_n,
\end{equation}

Since $\alpha$ is in $I(X,T)$, Lemma \ref{lemma 0} implies there exists a function  $f_1\in C(X,\ZZ)$ such that for every invariant probability measure $\mu$,
$$
\alpha=\int f_1 d\mu.
$$
On the other hand,
$$
\mu_n^Ta_n=\int f_2 d\mu, 
$$
where  $f_2=\sum_{i=1}^{C_n}a_n(i)1_{B_n(i)}$. Thus  Equation (\ref{eq1})   implies  that   for every invariant probability measure $\mu$,
 $$
\frac{1}{k}\mu_n^Tb_n=\int(f_1-f_2)d\mu=\int f d\mu,
$$
where $f=f_1-f_2$.
Thus we obtain,
\begin{align*}
\int kfd\mu & = \mu_n^T b_n = \int \sum_{i=1}^{C_n} b_n(i) 1_{B_n (i)} d\mu.
\end{align*}
Hence the map
$$
h = kf- \sum_{i=1}^{C_n} b_n(i) 1_{B_n (i)} 
$$ 
belongs to ${\rm Inf} (X,T)$ that is assumed to be trivial. 
Consequently, there exists a map $g\in C (X,\mathbb{Z} )$ such that $h = g - g\circ T$.

Choose $p\geq n$ such that $f$ is constant on any atom of the partition  $\mathcal{P}_p$ and such that the function $g$ is constant on  the base $B_p$.
This is always possible as the sequence $\{\mathcal{P}_j \}_{j}$ satisfies ({\bf KR1}-{\bf KR6}).
Let $x$ be an element in $B_{p}\subset  X$ and  let $1\leq i\leq C_p$ be such that $x\in B_p (i)$.  We have then

\begin{align*}
0 = & g(x) - g(T^{h_p(i)} x) = \sum_{j=0}^{h_p(i)-1} h(T^j x) = \sum_{j=0}^{h_p(i)-1} kf(T^j x) -  \sum_{j=0}^{h_p(i)-1} \sum_{l=1}^{C_n} b_n(l) 1_{B_n (l)} (T^j x)\\
 = & \left( k\sum_{j=0}^{h_p(i)-1} f(T^j x) \right) -\sum_{l=1}^{C_n}b_n(l)\sum_{j=0}^{h_p(i)-1}1_{B_n (l)} (T^j x)\\
= &  \left(k\sum_{j=0}^{h_p(i)-1} f(T^j x)\right) -\sum_{l=1}^{C_n}b_n(l)P_{p,n}(i,l)\\
= &  \left(k\sum_{j=0}^{h_p(i)-1} f(T^j x)\right)-(P_{p,n}b_n)(i).
\end{align*}

It follows that all the coordinates of  $P_{p,n}b_n$ are divisible by $k$.
On the other hand, for every $i\in I_p$ we have
$$
w_p(i)=(P_{p,n}w_n)(i)=P_{p,n}(ka_n+b_n)(i)=k(P_{p,n}a_n)(i)+(P_{p,n}b_n)(i),
$$
which contradicts that $I_p$ is non empty.
\end{proof}

\section{Examples}

In the sequel we will construct various examples of minimal Cantor systems starting with an ``abstract'' simple dimension group $(G,G^+ , u)$ having some fixed properties, and then we will make use of Theorem \ref{theo:hps} to have the existence of a minimal Cantor system $(X,T)$ having this prescribed simple dimension group.
Thus, we will most of the times avoid to mention this Theorem and we will identify $K^0 (X,T)$ to $(G,G^+ , u)$.

In some examples we will need some classical definition we will not recall and that can be found in any book on {\em Ergodic Theory}, we refer the reader to \cite{Petersen:1983} and \cite{Glasner:2003}.

\subsection{(Measurable) eigenvalues are not related with (strong) orbit equivalence}
At the difference of continuous additive  eigenvalues, irrational additive measurable eigenvalues can not be interpreted in terms of dimension group. Counter examples mainly come from the powerful result  obtained by N. Ormes in \cite{Ormes:1997} (Theorem 6.1), generalizing Jewett-Krieger Theorem to strong orbit equivalence classes.  

\begin{theorem}[\cite{Ormes:1997}]\label{Ormes97}
Let $(X,T)$ be a minimal Cantor systems and $\mu$ be an ergodic $S$-invariant Borel probability measure.
Let $(Y,S, \nu)$ be an ergodic measurable dynamical system of a non-atomic Lebesgue probability space $(Y,\nu )$ such that $\exp (2i\pi /p)$ is an eigenvalue of $(Y,S, \nu)$
for any element $\frac 1p \in \QQ (K^{0}(X,T))$.
Then, there exists a minimal Cantor system $(X,T')$ strongly orbit equivalent to $(X,T)$ such that $(X,S',\mu)$ is measurably conjugate to $(Y,S,\nu )$.
\end{theorem}

In the same paper Ormes obtained the following remarkable generalization of Dye's theorem \cite{Dye:1959}.
\begin{theorem}[\cite{Ormes:1997}]
\label{theo:ormes}
Let $(Y_1,S_1, \nu_1)$ and $(Y_2 , S_2 , \nu_2)$ be ergodic dynamical systems of non-atomic Lebesgue probability spaces.
There are minimal Cantor systems $(X,T_1)$ and $(X,T_2)$ and a Borel probability measure $\mu$ on $X$ which is $T_1$ and 
$T_2$ invariant such that:
\begin{enumerate}
\item
$(X, T_i , \mu )$ is measurably conjugate to $(Y_i , S_i, \nu_i)$, for $i=1,2$,
\item
$(X, T_1)$ is strongly orbit equivalent to $(X,T_2)$ by the identity map. 
\end{enumerate}
\end{theorem}

For example, one can take $(Y_1,S_1, \nu_1)$ and $(Y_2 , S_2 , \nu_2)$ with any countable groups (eventually trivial) $G_1$ and $G_2$ of measurable eigenvalues. 
The theorem asserts, there are strongly orbit equivalent minimal Cantor systems $(X,T_1)$ and $(X,T_2)$ that are measurably conjugate to the two previous dynamical systems, respectively.
Of course the groups of measurable eigenvalues of $(X,T_1)$ and $(X,T_2)$ are respectively $G_1$ and $G_2$.

\subsection{Rational eigenvalues are not preserved under orbit equivalence}
\label{sec:ratOE}  
Let $G=\ZZ\times \QQ$, $u=(1,1)$ and  $a\in (0,1)\cap \QQ$.
We set $G^+=\{v\in G: \tau_a(v)>0\}\cup \{0\}$, where $\tau_a(v)=av(1)+(1-a)v(2)$, for every $v=(v(1), v(2))\in G$. 
It is straightforward to check that $(G, G^+, u)$ is a simple dimension group verifying the following: 

\begin{itemize}
\item
$\mathbb{Q}(G,G^+,u) = \mathbb{Z}$, 
\item
$I(G,G^+, u) = \mathbb{Q}$, 
\item
$S(G,G^+,u)=\{\tau_a\}$  and 
\item
$G/{\rm Inf}(G)$, $I(G,G^+,u)$ and $\mathbb{Q}$ are isomorphic. 
\end{itemize}

From Theorem \cite{Herman&Putnam&Skau:1992} there exists minimal Cantor systems having $(G,G^+ , u)$ as a dimension group (up to isomorphism).

Proposition \ref{prop:eigincludedinimage} implies that $E(X,T) \subset I(X,T) = I(G , G^+ , u) = \mathbb{Q}$.
Using Theorem \ref{theo:gpso} one obtains $E(X,T) = \mathbb{Q}(G,G^+,u) = \mathbb{Z}$.
Thus, for every minimal Cantor system $(X,T)$ such that $K^0(X,T)$ is isomorphic to $(G,G^+,u)$, the group of eigenvalues $E(X,T)$ is equal to $\ZZ$.

Nevertheless $I(X,T)=\QQ$, so $I(X,T)/E(X,T)=\QQ/\ZZ$ is a torsion group. 
It is not difficult to see that ${\rm Inf}(K^0(X,T))$ is not trivial.

\begin{remark}
Thus, Theorem \ref{theo:main} is not true when ${\rm Inf}(K^0(X,T))$ is not trivial.
\end{remark}

As $(X,T)$ is uniquely ergodic, then, as observe in Section \ref{sec:infint},
$K^0(X,T)/{\rm Inf}(K^0(X,T))$ is isomorphic to $(I(X,T), I(X,T)\cap
\RR^+,1) = (\mathbb{Q} , \mathbb{Q}^+ , 1)$.

On the other hand, since every minimal Cantor system $(Y,S)$ whose simple dimension group is isomorphic to $(\QQ, \QQ^+,1)$ verifies $E(Y,S)=\QQ$ (due to Proposition \ref{prop:eigincludedinimage} and Theorem \ref{theo:gpso}) and is orbit equivalent to $(X,T)$ (by \cite[Theorem 2.2]{Giordano&Putnam&Skau:1995}), we deduce that the continuous rational eigenvalues are not invariant under orbit equivalence, unlike strong orbit equivalence. 

\begin{remark}
This example also shows that the groups of eigenvalues which are realizable among the class associated to $G$ (i.e., a class of strong orbit equivalence) are not necessarily realizable by systems in the class determined by $G/{\rm Inf}(G)$ (i.e., the corresponding class of orbit equivalence) and viceversa.
\end{remark}


\subsection{$I(X,T)/E(X,T)$ can be a torsion group, even if the rational subgroup of $I(X,T)$ is cyclic}
\label{sec:alpha}
Let $\alpha\in (0,1)$ be an irrational number such that $\alpha^2\in \ZZ+\alpha\ZZ$, e.g. the inverse of the golden mean. Let $(X,T)$ be the Sturmian subshift with angle $2\alpha$ (we refer to \cite{Morse&Hedlund:1940} for the definition). Let us recall that  its dimension group is isomorphic to $(\ZZ+2\alpha \ZZ,\ZZ+2\alpha \ZZ \cap \RR^{+}, 1)$.  Let $(Y,S)$ be any system such that its simple dimension group is isomorphic to $(\ZZ+\alpha\ZZ, ZZ+\alpha\ZZ \cap \RR^{+}, 1)$.
These groups having a unique trace, $(Y,S)$ and $(X,T)$ are uniquely ergodic. 
 
The rational subgroup of $\ZZ+\alpha\ZZ$ being $\mathbb{Z}$, Theorem \ref{theo:ormes} ensures that we can choose $(Y,S)$ having no non trivial measurable eigenfunctions; i.e., $(Y,S)$ is {\em weakly mixing}. 
We consider the product system $(X \times Y, T \times S)$ with the product action, i.e. $T\times S(x,y) = (Tx, Sy)$ for any $x\in X, y\in Y$. 

We will show that 
\begin{enumerate}
\item for the image subgroup, we have $I(X\times Y,T\times S)=  \ZZ+\alpha\ZZ$,

\item for the set of additive eigenvalues  $E(X\times Y, T\times S)=\ZZ+2\alpha\ZZ$.
\end{enumerate}
Thus the quotient group $I(X\times Y, T\times S)/ E(X\times Y,  T \times  S)$ will be isomorphic to $\ZZ/2\ZZ$.

We call $\pi_X$ and $\pi_Y$ the projections of $X\times Y$ on $X$ and $Y$ respectively, and $\mu$ and $\nu$ the unique invariant probability measures of $(X,T)$ and $(Y,S)$ respectively.  Let $\lambda$ be the Lebesgue measure in $\sS^1$. We denote $(\sS^{1}, R_{2\alpha})$ the rotation by angle $2\pi 2\alpha$ on the circle $\sS^{1}$.

Since $(Y,S)$ is weakly mixing, it is disjoint from $(\sS^1 , R_{2\alpha})$ (see Theorem 6.27 in \cite{Glasner:2003}) and thus,
the product measure $\lambda\times \nu$ is the unique invariant probability  measure of $(\sS^1\times Y, R_{2\alpha}\times S)$.
Since $(X,T)$ and $(\sS^1, R_{2\alpha})$ are measure theoretically conjugate, clearly $\mu\times \nu$ is the unique invariant probability measure of $(X\times Y, T\times S)$.
Thus $(X\times Y,T\times S)$ is uniquely ergodic. Notice moreover that every open set has a  positive measure. The Ergodic Theorem ensures then  that the system $(X\times Y, T\times S)$ is minimal.

One can checks, for example using Lemma 2.6 in \cite{Glasner&Weiss:1995}, that 
\begin{eqnarray*}
\{ \mu (A); \ A \hbox{ clopen subset of } X \} &= (\mathbb{Z} + 2 \alpha \mathbb{Z})\cap [0,1],  \textrm{ and} \\
\{ \nu (B); \ B \hbox{ clopen subset of } Y \} &= (\mathbb{Z} + \alpha \mathbb{Z})\cap [0,1].
\end{eqnarray*}
Also  notice that   any clopen set $C\subseteq X\times Y$ is a finite union of clopen sets of the kind $A\times B$, where $A$ and $B$ are clopen subsets of $X$ and $Y$ respectively. 
Hence, by the very definition of $\alpha$, we get
$$
I(X\times Y,T\times S)=\langle \{\mu(A)\nu(B): A\subseteq X, B\subseteq Y \mbox{ clopen subsets} \} \rangle =\ZZ+\alpha\ZZ.
$$

Fundamental properties of Sturmian subshifts ensure there exists $\phi:X\to \sS^1$ an almost 1-1 factor map from $(X,T)$ to $(\sS^1,R_{2\alpha})$.
The function $\phi\circ \pi_X$ is then a factor map of the product system. This shows that $2\alpha\in E(X\times Y,T\times S)$.


 
 We will show that $\alpha$ is not an additive eigenvalue of the system $(X\times Y, T\times S)$. 
 
Suppose there exists a continuous eigenfunction  $f  \colon X\times Y \to \sS^{1}$, such that $ f \circ  T\times S= e^{2i \pi \alpha} f$.
Since the map  $Id \times S$, product of the identity with the map $S$,  commutes with the product action, the map $f \circ Id \times S$ is also a continuous eigenfunction associated with the same  eigenvalue. So there is a constant $\lambda \in \sS^{1} $ such that $f \circ Id \times S = \lambda f$. It follows for any $x\in X$, the map $y \mapsto f(x,y)$ is a continuous eigenfunction of the system $(Y, S)$ associated to the eigenvalue $\lambda$.
The system $(Y,S)$  being weakly mixing, we get $\lambda =1$ and $f(x,y)$ does not depend on $y$, we denote this last value $f(x)$. So the map $x \mapsto f(x)$ is a continuous eigenfunction of the system $(X,T)$ associated with the eigenvalue $e^{2i \pi \alpha}$. This is impossible because this Sturmian subshift is an almost one-to-one extension of $(\sS^{1}, R_{2\alpha})$.

\medskip

We conclude that $E(X\times Y , T\times S) = \mathbb{Z} + 2\alpha \mathbb{Z}$.
 
\medskip


According to Theorem \ref{theo:main}, the infinitesimal subgroup of $K^0(X\times Y, T\times S)$ must be non-trivial. 
Let us give an example of a non-trivial infinitesimal element in the system $(X \times Y, T \times S)$.

\medskip 

From Lemma \ref{lemma 0}, there exists  a function $g \in C(Y, \ZZ)$ such that $\int g d \nu =2 \alpha$. 

{\it Claim.} For any function  $f \in C(X,\ZZ)$ such that $\int f d\mu = \int g d\nu$, the function $F : (x,y) \in X\times Y \mapsto f(x)-g(y) $ is a non trivial infinitesimal.


\medskip

A standard computation show us that $\int_{X \times Y} f(x)-g(y) d \mu \times \nu =0$.
Therefore it remains to prove that it is not a coboundary of $C(X\times Y , \ZZ)$.

Let us assume that the function $f-g$ is such a coboundary.
Then, there exists a function  $H \in C(X\times Y)$  such that 
$$
f(x)-g(y) = H(x,y)-H(Tx, Sy) \hspace{1cm} \forall x\in X, y\in Y.
$$

By taking the integral of the former equality for the measure $\mu$, we obtain 
\begin{eqnarray*}
 g(y) =&  \int f d\mu - \int H (x,y) d\mu(x) + \int H(Tx, Sy) d\mu(x)\\
  =& \int gd\nu -    \int H(x,y) d\mu(x) + \int H(x, Sy) d\mu(x).
\end{eqnarray*}

By the Lebesgue's dominated convergence Theorem, the function $h \colon y \mapsto \int H(x,y) d\mu(x)$ is continuous. So $g - \int g d\nu = h\circ S - h$  is a real coboundary. 
Then, $g$ taking integer values, the function $y \mapsto \textrm{exp }(2i\pi h(y))$ defines a continuous eigenfunction associated to the additive eigenvalue $-\int g d\nu = -2 \alpha$ for the system $(Y,S)$. 
This is impossible because this system is weakly mixing.  
This proves our claim.

\section{Some results about realization.}

\begin{definition}
Let $(G,G^+,u)$ be a simple dimension group with distinguished order unit. We define
$\E(G,G^+,u)$ as the collection of all the subgroups $\Gamma$ of $\RR$
for which there exists a minimal Cantor system $(X,T)$ such that
$K^0(X,T)$ and $E(X,T)$ are isomorphic to $(G,G^+,u)$ and $\Gamma$
respectively.
\end{definition}

In this section we are interested in a characterization of the family $\E(G,G^+,u)$, for a given simple dimension group.  Most of our results are based in \cite{Sugisaki:2011}.

\medskip

\begin{remark}{\rm 
Proposition \ref{prop:eigincludedinimage} implies that the elements in $\E(G,G^+,u)$ are subgroups of $I(G,G^+,u)$. If in addition $\inf(G)=\{0\}$,  from Theorem \ref{theo:main} we get the following:
$$
\E(G,G^+,u)\subseteq \{\Gamma: \mbox{ subgroup of  } I(G,G^+,u) \mbox{ such that } I(G,G^+,u)/\Gamma \mbox{ is torsion free }\}.
$$
}
\end{remark}

\subsection{Basic example} Let $\alpha$ be an irrational number. 
Consider $G=\ZZ+\alpha\ZZ$, $G^+=G\cap \RR^+$ and $u=1$. 
Since the infinitesimal subgroup of $(G,G^+,u)$ is trivial,  the collection   $\E(G,G^+,u)$ is a subfamily of  $\{\ZZ, \ZZ+\alpha\ZZ\}$.  
It is known that the dimension group associated to  the Sturmian subshift  with angle $\alpha$  is isomorphic to $(G,G^+,u)$ (see \cite{Morse&Hedlund:1940}  \cite{Dartnell&Durand&Maass:2000}).
Moreover, it is an almost 1-1 extension of the rotation with angle $\alpha$.
Hence its subgroup of eigenvalues equals $G$. 
On the other hand, Theorem \ref{Ormes97} implies  there exists a minimal Cantor system having no non trivial eigenvalues whose dimension group is isomorphic to $(G,G^+,u)$. Thus we get
$$
\E(G, G^+, u)=\{\ZZ, \ZZ+\alpha\ZZ\}.
$$


\subsection{Eigenvalues and dimension subgroups.}

Let $(G,G^+,u)$ and  $(H,H^+,w)$ be two simple dimension group with distinguished order unit. An {\it order embedding} is a monomorphism  $i:H\to G$ such that $i(h)\in G^+$ if and only if $h\in H^+$ and $i(w)=u$. An order embedding always induces an affine homomorphism $i^*: S(G)\to S(H)$ by 
$$
i^*(\tau)(h)=\tau(i(h)), \mbox{ for every } \tau\in S(G) \mbox{ and } h\in H.
$$

We use the next two lemmas to show that a minimal Cantor system and any of its almost 1-1 extensions share their maximal equicontinuous factor. 

The second lemma is a converse of the first one.  
A proof of the next result, in a more general context, can be
found in \cite[Chapter 9]{Auslander:1988}.
\begin{lemma}\label{proximal to eigenvalues}
Let $\pi:(X,T)\to (Y,S)$ be a proximal extension  of minimal
Cantor systems, then $(X,T)$ and $(Y,S)$ have the same maximal
equicontinuous factor.
\end{lemma}

\begin{lemma}\label{almost 1-1 to proximal}
Let $\pi: (X,T)\to (Y,S)$ be an almost 1-1 extension of  compact
systems, such that $(Y,S)$ is minimal. Then $\pi$ is a proximal
extension.
\end{lemma}

\begin{proof}
Let $y\in Y$ be an element having only one pre-image by $\pi$. If $\pi$ is injective, then the result
is trivial. We can assume then there exist $x'\neq x''$ in $X$
such that $\pi(x')=\pi(x'')=y'\in Y$. Suppose that $x'$ and $x''$
are not proximal. This means there exist $\varepsilon>0$ and
$n_0\in \NN$, such that for every $n\geq n_0$,
\begin{equation}\label{eq1-almost 1-1}
d(T^n(x'),T^n(x''))>\varepsilon.
\end{equation}
Since $(Y,S)$ is minimal, there exists a subsequence
$(S^{n_i}(y'))_{i\geq 0}$ of the orbit of $y'$ that converges to
$y$. By compactness, taking subsequences if needed, we can suppose
$(T^{n_{i}}(x'))_{i\geq 0}$ and $(T^{n_{i}}(x''))_{i\geq 0}$
converging to some  $z'$ and $z''$ respectively. Inequality
(\ref{eq1-almost 1-1}) ensures that $z'\neq z''$, and since $\pi$
is continuous, we have
$\lim_{i\to\infty}\pi(T^{n_{i}}(x'))=\pi(z')$ and
$\lim_{i\to\infty}\pi(T^{n_{i}}(x''))=\pi(z'')$. On the other
hand, the choice of $(n_{i})_{i\geq 0}$ implies that
$$\pi(z')=\lim_{i\to\infty}\pi(T^{n_{i}}(x'))=\lim_{i\to\infty}S^{n_{i}}(y')= y,$$
$$\pi(z'')=\lim_{i\to\infty}\pi(T^{n_{i}}(x''))=\lim_{i\to\infty}S^{n_{i}}(y'')= y,$$
which contradicts the fact that $y$ has only one pre-image.
\end{proof}

Let us recall a consequence of  \cite[Theorem 1.1]{Sugisaki:2011} (see Corollary 1.2 in \cite{Sugisaki:2011}).

\begin{theorem}\label{theorem-sugisaki}
Suppose that $(Y,S)$ is a uniquely ergodic minimal Cantor system and $(G, G^+,u)$ is a simple dimension group with distinguished order unit satisfying the following assumptions:
\begin{itemize}
\item[(i)] there is an order embedding $i: K^0(Y,S)\to (G, G^+,u)$,
\item[(ii)] $G/i(K^0(Y,S))$ is torsion free.
\end{itemize}
Then there exists a minimal Cantor system $(X,T)$ such that  $K^0(X,T)$ is isomorphic to $(G, G^+, u)$  and such that there is an almost one-to-one factor map $\pi:(X,T)\to (Y,S)$.
\end{theorem}

The next result is a direct consequence of  Lemma  \ref{proximal to
eigenvalues}, Lemma \ref{almost 1-1 to proximal} and Theorem \ref{theorem-sugisaki}, but stated in terms of dimension groups.

\begin{proposition}\label{Sugisaki}
Let $(G,G^+,u)$ and  $(H,H^+,w)$ be two simple dimension group with distinguished order unit, such that $(H, H^+,w)$ has a unique trace. Suppose there exists an order embedding $i:H\to G$ with $G/i(H)$  torsion free. Then $\E(H,H^+,w)$ is a subfamily of $\E(G,G^+,u)$.
\end{proposition}

Recall that for any countable dense subgroup $\Gamma$ of $\RR$ containing $\ZZ$, the dimension group   $(\Gamma, \Gamma\cap\RR^+, 1)$ has only one trace and no non trivial infinitesimal (see \cite{Effros:1981}). 
\begin{proposition}\label{prop19}
Let $(G,G^+,u)$ be a simple dimension group with  distinguished order unit  and with a trivial  infinitesimal subgroup. Then, for any subgroup $\Gamma$  of  $I(G, G^+,u)$  with $\ZZ\subseteq \Gamma$ and $I(G,G^+,u)/\Gamma$ torsion free, the family  $\E(\Gamma, \Gamma\cap \RR^+,1)$ is contained in  $\E(G, G^+,u)$.
 \end{proposition}

\begin{proof} Let $\tau: \tilde{G}\to I(G,G^+.u)$ be the  morphism given by Corollary \ref{cor:onto}. Since the infinitesimal subgroup is trivial, it is an isomorphism. Let $(H,H^+, u)$ be the dimension group image  of  $(\Gamma, \Gamma \cap\RR^+, 1)$ by the inverse $\tau^{-1}$. 
It is easy to check the group $G/\tilde{G}$ is torsion free, so by hypothesis, the quotient group $G/H$ is also torsion free. 
Moreover,  $S(\Gamma, \Gamma \cap \RR^+, 1)$ has only one element because $S(H,H^+,u)  \simeq S(\Gamma, \Gamma \cap \RR^+, 1)$.  The desired result follows from Proposition \ref{Sugisaki}.
\end{proof}

 As a consequence of these results, we obtain the following characterization.
\begin{corollary}\label{remark22}
The following are equivalent:
\begin{enumerate}
\item For any countable dense subgroup $\Gamma$  of $\RR$ containing $\ZZ$,   there is a Cantor minimal system  $(X,T)$ with  $E(X,T)= \Gamma$, and  $K^{0}(X,T) \simeq (\Gamma, \Gamma \cap \RR^+,1)$, i.e.
$$\Gamma\in \E(\Gamma, \Gamma \cap \RR^+,1). $$

\item For any  simple dimension group with distinguished order unit $(G,G^+,u)$ and with no non trivial  infinitesimal,
$$
\E(G,G^+,u)=\{\Gamma;  \Gamma \mbox{ is a subgroup of  } I(G,G^+,u) \mbox{ such that } I(G,G^+,u)/\Gamma \mbox{ is torsion free}\}.
$$
\end{enumerate}
\end{corollary}
\begin{proof}
 It is obvious that (1) is a consequence of (2). Proposition \ref{prop19} and (1) imply  (2). 
\end{proof}

\begin{proposition}\label{lemma 4}
Let $(G,G^+,u)$ be a simple dimension group with distinguished order unit and  let $\Gamma$ be a subgroup of $I(G,G^+,u)$  verifying the following: 
\begin{itemize}
\item $\Gamma$  is generated by a family of rationally independent numbers containing $1$.
 \item $I(G,G^+,u)/\Gamma$ is torsion free. 
\end{itemize}
Then  $\E(\Gamma, \Gamma\cap\RR^+,1)$ is contained in $\E(G,G^+,u)$.
 \end{proposition}
{\rm   Remark that Proposition \ref{lemma 4} includes the case where $\Gamma$ is finitely generated.  Example \ref{sec:ratOE} shows that Proposition \ref{lemma 4}  becomes false whenever the group  $\Gamma$  is not generated by a rationally independent family.}

\begin{proof}
Let $\{\alpha_i\}_{i\geq 0}$ be a rationally independent family generating $\Gamma$. Without loss of generality, we can assume that  $\alpha_{0}=1$.  From Lemma \ref{lemma 0}, for every $i\in \NN$, there exists a $g_i\in G$ such that $\tau(g_i)=\alpha_i$, for every $\tau\in S(G,G^+,u)$. We choose $g_{0} = u$. 

For  any $\alpha\in \Gamma$, there exists a unique sequence of integers $(m_i)_{i\geq 0}$ with   $m_i=0$ except for a finite number of $i$'s  such that  $\alpha=\sum_{i\geq 0}m_i\alpha_i$. This implies that the function $\phi: \Gamma\to G$ given by $\phi(\alpha)=\sum_{i\geq 0}m_ig_i$ is a well defined one-to-one  homomorphism. It is not difficult to see that this is an order embedding  such that $G/\phi(\Gamma)$ is  torsion free. Since $(\Gamma, \Gamma\cap\RR^+, 1)$ has only one trace, from Proposition \ref{Sugisaki}  we obtain the desired property.
\end{proof}

\subsection{Necessary conditions for $\Gamma\in \E(\Gamma,\Gamma\cap\RR^+,1)$. }
Let $(G,G^+,u)$ be a simple dimension group with distinguished order unit having no non trivial infinitesimal and with a unique trace. So it is isomorphic to the dimension group $(\Gamma, \Gamma \cap \RR^+, 1)$, where $\Gamma=I(G,G^+,u)$.  We suppose here, that $\Gamma\in \E(\Gamma,\Gamma \cap \RR^+,1)$. That is, there exists a uniquely ergodic minimal Cantor system $(X,T)$ whose dimension group is isomorphic to $(\Gamma, \Gamma \cap \RR^+,1)$ and such that $E(X,T)=\Gamma$.

We use the notations of  the sections \ref{subsec:partandto} and \ref{sec:necessarycond}. Recall that for every $n\geq 1$,  $\mu_n=(\mu_n(k))_{k=1}^{C_n}$ denotes the vector of measures of the bases of the partition $\P_n$, corresponding to the unique invariant probability measure $\mu$ of $(X,T)$  (see Section \ref{subsec:partandto} for definitions). Since $(X,T)$ is uniquely ergodic, the group $\Gamma$ is generated by $\{\mu_n(k): 1\leq k\leq C_n, n\geq 1\}$, which implies that for every $m\geq 1$ and every $1\leq k\leq C_m$ the number $\mu_{m,k}$ is in $E(X,T)$. From Lemma \ref{lemma:thelemma} we get that for every $n$ sufficiently large, there exist an integer vector $w_n$ and a real vector $v_n$ such that
\begin{equation}\label{limite}
P_nH_1\mu_{m}(k)=v_n+w_n \mbox{ and }\sum_{l\geq n}\| P_{l,n}v_n\|_{\infty}<\infty.
\end{equation}
Multiplying by $\mu_n^T$ the first equation we get but the normalization conditions (see Equation \eqref{eq:vectormeasure} and Lemma \ref{lemma:orthogonalite})
$$
\mu_{m}(k)=\mu_n^Tw_n.
$$
On the other hand, we have 
$$
\mu_{m}(k)=\mu_n^TP_{n,m}(\cdot,k).
$$
Since the infinitesimal subgroup is trivial, the previous two equations implies that for $n$ sufficiently large 
\begin{equation}\label{igualdad}
w_n=P_{n,m}(\cdot,k).
\end{equation}
Equations (\ref{limite}) and (\ref{igualdad})  imply
\begin{equation}\label{eqe1}
\sum_{n\geq 1}\max_{i}\|h_{n,i}\mu_{m}-P_{n,m}^T(\cdot,i)\|_{\infty}<\infty.
\end{equation}
On the other hand, the unique ergodicity of the system $(X,T)$ implies that the rows of the matrix $P_{n,m}$ converges with $n$ (after normalization) to $\mu_m$. That is
$$
\lim_{n\to \infty}\max_i\left \|\mu_m-\frac{1}{h_{n,i}}P_{n,m}^T(\cdot,i) \right \|_{\infty}=0.
$$
Thus from (\ref{eqe1}), we deduce  that if $E(X,T)=\Gamma$, then the rate of convergence of the rows of  $P_{n,m}$  to the direction generated by $\mu_m$ has to be extremely fast. 



\bibliographystyle{plain}
\bibliography{CDP2-bis}

\begin{thebibliography}{10}

\bibitem{Auslander:1988}
J.~Auslander.
\newblock {\em Minimal flows and their extensions}, volume 153 of {\em
  North-Holland Mathematics Studies}.
\newblock North-Holland Publishing Co., Amsterdam, 1988.
\newblock Notas de Matem{\'a}tica [Mathematical Notes], 122.

\bibitem{Boyle&Handelman:1994}
M.~Boyle and D.~Handelman.
\newblock Entropy versus orbit equivalence for minimal homeomorphisms.
\newblock {\em Pacific J. Math.}, 164:1--13, 1994.

\bibitem{Bressaud&Durand&Maass:2005}
X.~Bressaud, F.~Durand, and A.~Maass.
\newblock Necessary and sufficient conditions to be an eigenvalue for linearly
  recurrent dynamical cantor systems.
\newblock {\em J. London Math. Soc.}, 72:799--816, 2005.

\bibitem{Bressaud&Durand&Maass:2010}
X.~Bressaud, F.~Durand, and A.~Maass.
\newblock On the eigenvalues of finite rank {B}ratteli-{V}ershik dynamical
  systems.
\newblock {\em Ergod. Th. \& Dynam. Sys.}, 30:639--664, 2010.

\bibitem{Cortez&Durand&Host&Maass:2003}
M.~I. Cortez, F.~Durand, B.~Host, and A.~Maass.
\newblock Continuous and measurable eigenfunctions of linearly recurrent
  dynamical cantor systems.
\newblock {\em J. London Math. Soc.}, 67:790--804, 2003.

\bibitem{Dartnell&Durand&Maass:2000}
P.~Dartnell, F.~Durand, and A.~Maass.
\newblock Orbit equivalence and kakutani equivalence with sturmian subshifts.
\newblock {\em Studia Math.}, 142:25--45, 2000.

\bibitem{Durand&Frank&Maass:2014}
F.~Durand, A.~Frank, and A.~Maass.
\newblock Eigenvalues of toeplitz minimal systems of finite topological rank.
\newblock {\em to appear in Ergod. Th. {\&} Dynam. Sys.}, 2014.

\bibitem{Dye:1959}
H.~Dye.
\newblock On groups of measure preserving transformations {I}.
\newblock {\em Amer. J. Math.}, 81:119--159, 1959.

\bibitem{Effros:1981}
E.~G. Edward.
\newblock {\em Dimensions and {$C^{\ast} $}-algebras}, volume~46 of {\em CBMS
  Regional Conference Series in Mathematics}.
\newblock Conference Board of the Mathematical Sciences, Washington, D.C.,
  1981.

\bibitem{Effros&Handelman&Shen:1980}
E.~Effros, D.~Handelman, and C.-L. Shen.
\newblock Dimension groups and their affine representations.
\newblock {\em Amer. J. Math.}, 102:385--407, 1980.

\bibitem{Exel:1987}
R.~Exel.
\newblock Rotation numbers for automorphisms of {$C^\ast$} algebras.
\newblock {\em Pacific J. Math.}, 127:31--89, 1987.

\bibitem{Giordano&Putnam&Skau:1995}
T.~Giordano, I.~Putnam, and C.~F. Skau.
\newblock Topological orbit equivalence and ${C}^{*}$-crossed products.
\newblock {\em Internat. J. Math.}, 469:51--111, 1995.

\bibitem{Glasner:2003}
E.~Glasner.
\newblock {\em Ergodic theory via joinings}, volume 101 of {\em Mathematical
  Surveys and Monographs}.
\newblock American Mathematical Society, Providence, RI, 2003.

\bibitem{Glasner&Weiss:1995}
E.~Glasner and B.~Weiss.
\newblock Weak orbit equivalence of cantor minimal systems.
\newblock {\em Internat. J. Math.}, 6:559--579, 1995.

\bibitem{Herman&Putnam&Skau:1992}
R.~H. Herman, I.~Putnam, and C.~F. Skau.
\newblock Ordered bratteli diagrams, dimension groups and topological dynamics.
\newblock {\em Internat. J. Math.}, 3:827--864, 1992.

\bibitem{itza-ortiz:2007}
B.~Itz{\'a}-Ortiz.
\newblock Eigenvalues, ${K}$-theory and minimal flows.
\newblock {\em Canad. J. Math.}, 59:596--613, 2007.

\bibitem{Krieger:1969}
W.~Krieger.
\newblock On non-singular transformations of a measure space. {I}, {II}.
\newblock {\em Z. Wahrscheinlichkeitstheorie und Verw. Gebiete 11 (1969),
  83-97; ibid.}, 11:98--119, 1969.

\bibitem{Krieger:1976}
W.~Krieger.
\newblock On ergodic flows and the isomorphism of factors.
\newblock {\em Math. Ann.}, 223:19--70, 1976.

\bibitem{Morse&Hedlund:1940}
M.~Morse and G.~A. Hedlund.
\newblock Symbolic dynamics {II}. {Sturmian} trajectories.
\newblock {\em Amer. J. Math.}, 62:1--42, 1940.

\bibitem{Murray&vonNeumann:1936}
F.~J. Murray and J.~Von~Neumann.
\newblock On rings of operators.
\newblock {\em Ann. of Math. (2)}, 37:116--229, 1936.

\bibitem{Ormes:1997}
N.~Ormes.
\newblock Strong orbit realization for minimal homeomorphisms.
\newblock {\em J. Anal. Math.}, 71:103--133, 1997.

\bibitem{Ornstein&Weiss:1980}
D.~Ornstein and B.~Weiss.
\newblock Ergodic theory of amenable group actions. {I}. {T}he {R}ohlin lemma.
\newblock {\em Bull. Amer. Math. Soc. (N.S.)}, 2:161--164, 1980.

\bibitem{Packer:1986}
J.~Packer.
\newblock {$K$}-theoretic invariants for {$C^\ast$}-algebras associated to
  transformations and induced flows.
\newblock {\em J. Funct. Anal.}, 67:25--59, 1986.

\bibitem{Petersen:1983}
K.~Petersen.
\newblock {\em Ergodic theory}, volume~2 of {\em Cambridge Studies in Advanced
  Mathematics}.
\newblock Cambridge University Press, Cambridge, 1983.

\bibitem{Renault:2009}
J.~Renault.
\newblock {\em {$C^\star$}-algebras and dynamical systems}.
\newblock Publica\c c\~oes Matem\'aticas do IMPA. [IMPA Mathematical
  Publications]. Instituto Nacional de Matem\'atica Pura e Aplicada (IMPA), Rio
  de Janeiro, 2009.
\newblock 27${^{{}}{\rm{o}}}$ Col{\'o}quio Brasileiro de Matem{\'a}tica. [27th
  Brazilian Mathematics Colloquium].

\bibitem{Riedel:1982}
N.~Riedel.
\newblock Classification of the {$C^{\ast} $}-algebras associated with minimal
  rotations.
\newblock {\em Pacific J. Math.}, 101:153--161, 1982.

\bibitem{Schwartzman:1957}
S.~Schwartzman.
\newblock Asymptotic cycles.
\newblock {\em Ann. of Math. (2)}, 66:270--284, 1957.

\bibitem{Singer:1955}
I.~M. Singer.
\newblock Automorphisms of finite factors.
\newblock {\em Amer. J. Math.}, 77:117--133, 1955.

\bibitem{Sugisaki:2003}
F.~Sugisaki.
\newblock The relationship between entropy and strong orbit equivalence for the
  minimal homeomorphisms. {I}.
\newblock {\em Internat. J. Math.}, 14:735--772, 2003.

\bibitem{Sugisaki:2011}
F.~Sugisaki.
\newblock Almost one-to-one extensions of {C}antor minimal systems and order
  embeddings of simple dimension groups.
\newblock {\em M\"unster J. Math.}, 4:141--169, 2011.

\bibitem{Tomiyama:1987}
J.~Tomiyama.
\newblock {\em Invitation to {$C^*$}-algebras and topological dynamics},
  volume~3 of {\em World Scientific Advanced Series in Dynamical Systems}.
\newblock World Scientific Publishing Co., Singapore, 1987.

\end{thebibliography}

\end{document}